\documentclass[12pt]{amsart}
\usepackage{graphicx}
\usepackage[headings]{fullpage}
\usepackage{bm,amssymb,epic,eepic,epsfig,amsbsy,amsmath,amscd}
\numberwithin{equation}{section}
                        \theoremstyle{plain}
\usepackage{mathrsfs}

\newcommand\no[1]{}

\newtheorem{theorem}{Theorem}[section]
\newtheorem{thm}{Theorem}
\newtheorem{lemma}[theorem]{Lemma}

\newtheorem{proposition}[theorem]{Proposition}

\theoremstyle{definition}
\newtheorem{remark}[theorem]{Remark}

\newcommand{\lcr}{\raisebox{-5pt}{\mbox{}\hspace{1pt}
                  \epsfig{file=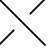}\hspace{1pt}\mbox{}}}

\def\BC{\mathbb C}
\def\BN{\mathbb N}
\def\BZ{\mathbb Z}
\def\BR{\mathbb R}

\def\fb{\mathfrak b}

\def\la{\langle}
\def\ra{\rangle}

\DeclareMathOperator{\tr}{\mathrm tr}

\def\be { \begin{equation} }
\def\ee { \end{equation} }

\begin{document}

\title[LO of cyclic branched covers]
{Left orderability of cyclic branched covers of rational knots 
$\bm{C(2n+1,2m,2)}$}

\author{Bradley Meyer}
\author{Anh T. Tran}

\begin{abstract}
We compute the nonabelian $\mathrm{SL_2}(\BC)$-character varieties of the rational knots $C(2n+1,2m,2)$ in the Conway notation, where $m$ and $n$ are non-zero integers. By studying real points on these varieties, we determine the left orderability of the fundamental groups of the cyclic branched covers of $C(2n+1,2m,2)$.
\end{abstract}



\address{Department of Mathematical Sciences, The University of Texas at Dallas, 
Richardson, TX 75080, USA}
\email{Bradley.Meyer@utdallas.edu}

\address{Department of Mathematical Sciences, The University of Texas at Dallas, 
Richardson, TX 75080, USA}
\email{att140830@utdallas.edu}

\maketitle

\section{Introduction}

We consider an important class of knots/links called rational knots/links. These are also known as two-bridge knots/links. In the Conway notation, a rational knot/link corresponds to a continued fraction 
\[
[a_1, a_2, \dots, a_k] =
a_1 + \dfrac{1}{a_2 + \dfrac{1}{\ddots \dfrac{1}{a_k}}}
\]
and it is denoted by $C(a_1,a_2, \ldots, a_k)$. This knot/link is the two-bridge knot/link $\fb(p,q)$ in the Schubert notation, where $\frac{p}{q} = [a_1, a_2, \dots, a_k]$, and so it is a knot if $p$ is odd and is a link if $p$ is even. Its knot/link diagram 
 is depicted as in Figure \ref{fig:rational}, where $a_i$ denotes the number of twists with sign in the $a_i$ box. Here the sign of the twist $\lcr$ in the $a_i$ box  is positive/negative for odd/even $i$.

\begin{figure}[h]
	\centering
	\includegraphics[scale=1]{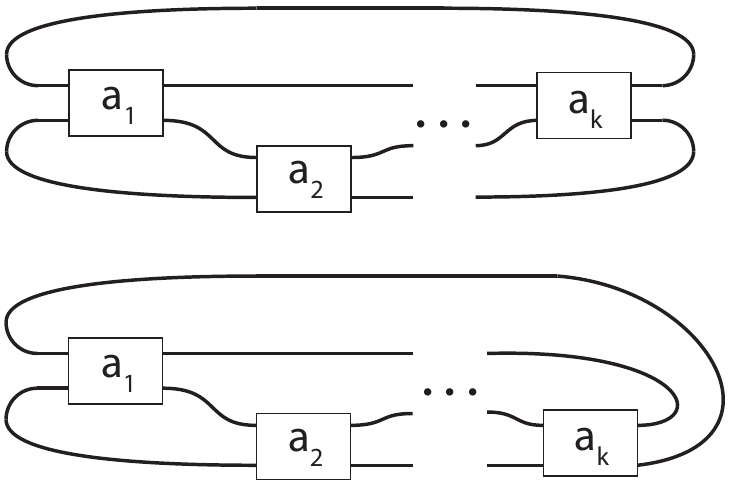}
	\caption{Rational knot/link $C(a_1,a_2, \ldots,a_k)$. The upper/lower one corresponds to odd/even $k$.}
	\label{fig:rational}
\end{figure}

\no{We will compute the nonabelian $\mathrm{SL_2}(\BC)$-character varieties of the rational knots $C(2n+1,2m,2)$ in the Conway notation. By studying real points on these varieties, we wil determine the left orderability of the fundamental groups of the cyclic branched covers of $C(2n+1,2m,2)$.
}

A non-trivial group $G$ is left orderable if there is a total ordering $<$ on $G$ such that $g < h$ implies $fg < fh$ for every $f, g, h \in G$. A motivation for studying left orderable groups in topology is a conjectured connection with L-spaces. An L-space is a rational homology sphere $M$ with rank of $\mathrm{\widehat{HF}} (M)$ equal to $| H_1 (M, \BZ) |$ \cite{OS}. The L-space conjecture of Boyer-Gordon-Watson \cite{BGW}  states that an irreducible rational homology 3-sphere is an L-space if and only if its fundamental group is not left orderable. 

We consider cyclic branched covers of knots in $S^3$ and study the left orderability of their fundamental groups. A sufficient condition for the fundamental group of the $r$-th cyclic branched cover of a prime knot to be left orderable was given in \cite{BGW, Hu}. As an application, it was proved that for any rational knot $K$ with non-zero signature the fundamental group of the $r$-th cyclic branched cover of $K$ is left orderable for sufficiently large $r$, see \cite{Hu, Tr-sign, Go}. For rational knots $C(k,2l)$, the left orderability of the fundamental groups of their cyclic branched covers was determined in \cite{DPT, Tr-DTK, Tu}. Moreover, Turner \cite{Tu} also determined the left orderability of the fundamental groups of the cyclic branched covers of the rational knots $C(2n+1,2,2)$ for positive integers $n$. 

In this paper, we will generalize Turner's result to all rational knots $C(2n+1,2m,2)$ where $m$ and $n$ are non-zero integers.  By studying real points on the nonabelian $\mathrm{SL_2}(\BC)$-character varieties of knot groups, we will prove the following.

\begin{thm} \label{thm2}
	The fundamental group of the $r$-th cyclic branched cover of the rational knot $C(2n+1, 2m, 2)$ is left orderable if
	\begin{enumerate}
		\item $r \ge 3$ when $n \ge 3$ or $n \le -4$.
		\item $r \ge 4$ when $n =2$ or $n =-3$.
		\item $r \ge 5$ when $n=1$ and $m=1,2$, or $n=-2$ and $m=-1$.
		\item $r \ge 6$ when $n=1$ and $m \ge 3$, or $n=-2$ and $m \le -2$.
		\item $r \geq 7$ when $n=1$ and $m \leq -4$ or $n = -2$ and $m \geq 6$.
		\item $r \geq 8$ when $n=1$ and $m = -2,-3$.
		\item $r \geq 9$ when $n=1$ and $m = -1$.
	\end{enumerate}
\end{thm}

\begin{figure}[h]
	\centering
	\includegraphics[scale=1]{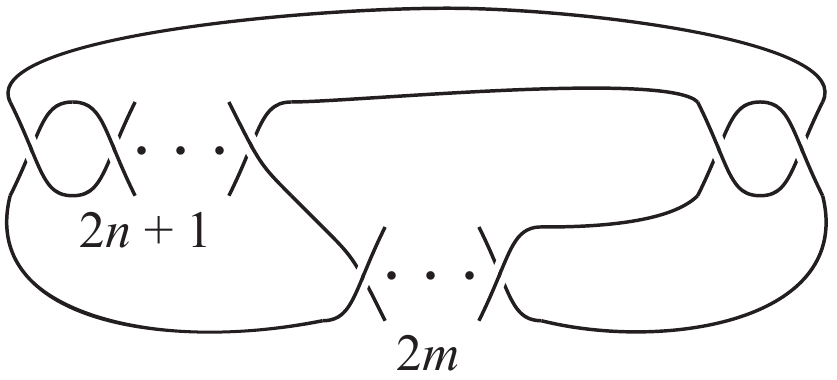}
	\caption{Rational knot $C(2n+1, 2m, 2)$.}
	\label{fig:Knot}
\end{figure}

This  paper is organized as follows. In Section \ref{Riley}, we compute the nonabelian $\mathrm{SL_2}(\BC)$-character varieties of the rational knots $C(2n+1, 2m, 2)$.  In Section \ref{lem}, we prove some properties of the Chebychev polynomials of the second kind and  character varieties of $C(2n+1, 2m,2)$. Finally, in Section \ref{LO} we study real points on these character varieties and give a proof of Theorem \ref{thm2}.

\section{Character varieties} \label{Riley}

In this section we will compute the nonabelian $\mathrm{SL_2}(\BC)$-character variety, i.e. the Riley polynomial, of $C(2n+1, 2m, 2)$ for non-zero integers $m$ and $n$. 

For a knot $K$ in $S^3$ we denote by $G(K)$ the knot group of $K$, which is the fundamental group of the knot complement $X_K = S^3 \setminus K$. 

\subsection{Knot group}  Let $K_{n,m}$ denote the rational knot $C(2n+1,2m,2)$. 

\begin{proposition} \label{prop10}
We have
$$G(K_{n,m})= \la a, b \mid wa =  b w \ra$$
where $w= v^n (aba^{-1}b^{-1})^m ab$ and $v = (aba^{-1}b^{-1})^m a (aba^{-1}b^{-1})^{-m} b$.
\end{proposition}

\begin{proof}
	\begin{figure}[h]
	\centering
	\includegraphics[scale=1]{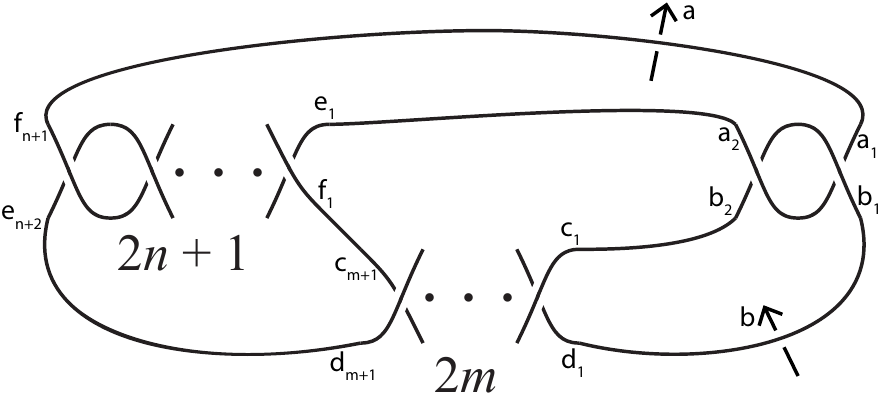}
	\caption{Generators $a,b$ of the knot group.}
	\label{fig:KnotGroup}
\end{figure}

Starting from the right hand side of the knot diagram, we have the following two relations:
\begin{align*}
	a_1 &= b_1 a_2 b_1^{-1}, \\
	b_1 &= a_2 b_2 a_2^{-1}.
\end{align*}

In the middle section of $2m$ crossings we have, by induction,
\begin{align*}
	c_{m+1} &= (d_1 c_1^{-1})^{-m} c_1 (d_1 c_1^{-1})^m, \\
	d_{m+1} &= (d_1 c_1^{-1})^{-m} d_1 (d_1 c_1^{-1})^m.
\end{align*}

In the left hand section of $2n + 1$ crossings we have, by induction,
\begin{align*}
	e_{n+1} &= (e_1 f_1)^{-n} e_1 (e_1 f_1)^n, \\
	f_{n+1} &= (e_1 f_1)^{-n} f_1 (e_1 f_1)^n.
\end{align*}

We have the following identifications:
\begin{gather*}
	a_1 = f_{n+1}, \\
	e_1 = a_2, \\
	f_1 = c_{m+1}, \\
	d_1 = b_1, \\
	c_1 = b_2.
\end{gather*}

Let $a_1 = a$ and $b_1 = b$. Using the identity $a = f_{n+1}$ and the relations listed above, we have
\[
	a = (a_2 c_{m+1})^{-n} (a^{-1} b^{-1} a b)^{-m} b^{-1} a^{-1} b a b (a^{-1} b^{-1} a b)^{m} (a_2 c_{m+1})^{n}.
\]
This implies that $wa = bw$ where
\[
	w = ab(a^{-1}b^{-1})^m(a_2 c_{m+1})^n.
\]
Writing $a_2 c_{m+1}$ in terms of $a$ and $b$, we have
\[
	a_2 c_{m+1} = b^{-1}ab(a^{-1}b^{-1}ab)^{-m} b^{-1}a^{-1}bab (a^{-1}b^{-1}ab)^m.
\]
Now we have
\begin{align*}
	w &= ab(a^{-1}b^{-1}ab)^m (b^{-1}ab(a^{-1}b^{-1}ab)^{-m} b^{-1}a^{-1}bab (a^{-1}b^{-1}ab)^m)^n \\
	&= (ab (a^{-1}b^{-1}ab)^m b^{-1} ab (a^{-1}b^{-1}ab)^{-m} b^{-1} a^{-1}b)^n (aba^{-1}b^{-1})^m ab \\
	&= ((aba^{-1}b^{-1})^m a (aba^{-1}b^{-1})^{-m} b)^n (aba^{-1}b^{-1})^m ab.
\end{align*}
This completes the proof of Proposition \ref{prop10}. 
\end{proof}

\subsection{Nonabelian representations} 

Suppose $\rho\colon G(K_{n,m})\to  \mathrm{SL}_2(\BC)$ is a nonabelian representation. 
Up to conjugation, we may assume that 
\begin{equation} \label{repn}
\rho(a) = \left[
\begin{array}{cc}
t & 1\\
0 & t^{-1}
\end{array}
\right] \quad \text{and} \quad \rho(b) = \left[
\begin{array}{cc}
t & 0\\
-u & t^{-1}
\end{array}
\right]
\end{equation}
where $(t,u) \in \BC^2$ satisfies the matrix equation $\rho(wa) = \rho(bw)$. It is known that this matrix equation is equivalent to a single polynomial equation $R_{K_{n,m}}(x,y) =0$, where $x= \tr \rho(a)=t+t^{-1}$, $y=\tr \rho(ab^{-1})=u+2$ and $R_K(x,y)$ is the Riley polynomial of a rational knot $K$, see \cite{Ri}. This polynomial can be computed via the formula $R_K = w_{11} + (t^{-1} - t)w_{12}$, where $w_{ij}$ is the $ij$-th entry of the matrix $\rho(w)$. For the rational knot $K_{n,m} = C(2n+1, 2m,2)$, the Riley polynomial can be described via the Chebyshev polynomials as follows. 

Let $S_k(z)$ be the Chebyshev polynomials of the second kind defined by $S_0(z)=1$, $S_1(z)=z$ and $S_{k}(z) = z S_{k-1}(z) - S_{k-2}(z)$ for all integers $k$. 

The following lemmas are elementary, see e.g \cite{Tr-torsion}.

\begin{lemma} \label{che}
For any integer $k$ we have $$S^2_k(z) + S^2_{k}(z) - z S_{k}(z) S_{k-1}(z) = S^2_k(z) - S_{k+1}(z) S_{k-1}(z)=1.$$
\end{lemma}

\begin{lemma} \label{power}
Suppose $V = \left[ \begin{array}{cc}
e & f \\
g & h \end{array} \right] \in \mathrm{SL}_2(\BC)$ and $z= \tr V.$ For any integer $k$ we have 
\begin{eqnarray*}
V^k &=& \left[ \begin{array}{cc}
S_{k}(z) - h S_{k-1}(z) & f S_{k-1}(z) \\
g S_{k-1}(z) & S_{k}(z) - e S_{k-1}(z) \end{array} \right].
\end{eqnarray*}
\end{lemma}

\no{
Suppose $\rho: G(K_{n,m})\to SL_2(\BC)$ is a nonabelian representation. 
Up to conjugation we may assume that $$\rho(a) = \left[
\begin{array}{cc}
t & 1\\
0 & t^{-1}
\end{array}
\right] \quad \text{and} \quad \rho(b) = \left[
\begin{array}{cc}
t & 0\\
-u & t^{-1}
\end{array}
\right].$$ 
The Riley polynomial is $R = w_{11} + (t^{-1} - t)w_{12}$.
}

Let $\alpha = \tr \rho(aba^{-1}b^{-1})$ and $\beta = \tr \rho(v)$.

\begin{proposition} \label{ab}
We have
\begin{eqnarray*}
\alpha &=& y^2 - x^2 y + 2x^2-2,\\
\beta &=& 2 + (x^2-y-2)\big( S_m(\alpha)+(1-y)S_{m-1}(\alpha) \big)^2.
\end{eqnarray*}
\end{proposition}

\begin{proof}
By a direct calculation we have 
$$\rho(aba^{-1}b^{-1}) = \left[
\begin{array}{cc}
 u^2-u(t^2+t^{-2}-1)+1 & t(u-t^2+1)\\
 t^{-1}u(u-t^{-2}+1) & u+1 \\
\end{array}
\right].$$
Hence
$$\alpha = \tr \rho(aba^{-1}b^{-1}) = u^2-u(t^2+t^{-2}-2)+2.$$

By Lemma \ref{power} we have $\rho((aba^{-1}b^{-1})^m) = \left[
\begin{array}{cc}
 e & f \\
 g & h \\
\end{array}
\right]$ where
\begin{eqnarray*}
e &=& S_m(\alpha) - (u+1)S_{m-1}(\alpha), \\
f &=& t(u-t^2+1)S_{m-1}(\alpha), \\
g &=& t^{-1}u(u-t^{-2}+1)S_{m-1}(\alpha), \\
h &=& S_m(\alpha) - \big( u^2-u(t^2+t^{-2}-1)+1 \big)S_{m-1}(\alpha).
\end{eqnarray*}
By a direct calculation we have
\begin{eqnarray*}
\rho(v) &=& \rho \big( (aba^{-1}b^{-1})^m a (aba^{-1}b^{-1})^{-m} b \big) \\
           &=& \left[
\begin{array}{cc}
 e & f \\
 g & h \\
\end{array}
\right] \left[
\begin{array}{cc}
t & 1\\
0 & t^{-1}
\end{array}
\right] \left[
\begin{array}{cc}
 h & -f \\
 -g & e \\
\end{array}
\right] \left[
\begin{array}{cc}
t & 0\\
-u & t^{-1}
\end{array}
\right]
           = \left[
\begin{array}{cc}
 v_{11} & v_{12} \\
 v_{21} & v_{22} \\
\end{array}
\right]
\end{eqnarray*} where
\begin{eqnarray*}
v_{11} &=& t^2eh-teg-fg+ue((t-t^{-1})f-e), \\
v_{12} &=& -t^{-1}e((t-t^{-1})f-e), \\
v_{21} &=& tg((t-t^{-1})h-g)+u(tfg-t^{-1}eh-eg), \\
v_{22} &=& t^{-1}eg+t^{-2}eh-fg.
\end{eqnarray*}

With $\beta = \tr \rho(v) = v_{11} + v_{22}$ we have
\begin{eqnarray*}
\beta &=& (t^2 + t^{-2})eh -(t-t^{-1})eg-2fg  + ue((t-t^{-1})f-e) \\
      &=& 2(eh-fg) + e \big[ (t^2+t^{-2}-2)h +(t-t^{-1}) (uf-g)- u e\big].
\end{eqnarray*}
Since $eh-fg =1$ and $
uf -g =  (t-t^{-1}) u (u-t^2-t^{-2}) S_{m-1}(\alpha) = (t-t^{-1})(e-h),
$
we obtain $$\beta =  2 +(t^2+t^{-2}-2-u)e^2 = 2 + (t^2+t^{-2}-2-u) \big( S_m(\alpha)-(u+1)S_{m-1}(\alpha) \big)^2.$$ 

Finally, since $t^2 + t^{-2} = x^2-2$ and $u=y-2$, Proposition \ref{ab} follows.
\end{proof}

Let $R_{n,m}(x,y)$ denote the Riley polynomial of $C(2n+1, 2m, 2)$.

\begin{proposition} \label{R}
We have
$$
R_{n,m}(x,y) = \big( (x^2-y-1)S_m(\alpha)-S_{m-1}(\alpha) \big) S_n(\beta) - \big( (x^2-y-1)S_{m-1}(\alpha)-S_{m-2}(\alpha) \big) S_{n-1}(\beta).
$$
\end{proposition}

\begin{proof}
By Lemma \ref{power} we have 
$\rho(v^n) = \left[
\begin{array}{cc}
 p & q \\
 r & s \\
\end{array}
\right]$ where 
\begin{eqnarray*}
p &=& S_n(\beta) - v_{22} S_{n-1}(\beta), \qquad q = v_{12} S_{n-1}(\beta),\\
r &=& v_{21} S_{n-1}(\beta), \qquad \qquad \quad \,\,\,\, s = S_n(\beta) - v_{11} S_{n-1}(\beta).
\end{eqnarray*}
By a direct calculation we have 
\begin{eqnarray*}
\rho(w) &=&  \rho(v^n (aba^{-1}b^{-1})^m ab) \\
            &=& \left[
\begin{array}{cc}
 p & q \\
 r & s \\
\end{array}
\right] \left[
\begin{array}{cc}
 e & f \\
 g & h \\
\end{array}
\right] \left[
\begin{array}{cc}
t & 1\\
0 & t^{-1}
\end{array}
\right] \left[
\begin{array}{cc}
t & 0\\
-u & t^{-1}
\end{array}
\right]
            = \left[
\begin{array}{cc}
 w_{11} & w_{12} \\
 w_{21} & w_{22} \\
\end{array}
\right]
\end{eqnarray*} 
where
\begin{eqnarray*}
w_{11} &=& (t^2-u)(ep + gq)-t^{-1}u(fp + hq),\\
w_{12} &=& t^{-2}(fp + hq) + t^{-1}(ep + gq),\\
w_{21} &=& (t^2-u)(er + gs)-t^{-1}u(fr + hs),\\
w_{22} &=& t^{-2}(fr + hs) + t^{-1}(er + gs).
\end{eqnarray*}
The Riley polynomial is
$$
R_{n,m} = w_{11} + (t^{-1} - t)w_{12} = (t^2+t^{-2}-u-1)(ep + gq)-t^{-1}(u+1-t^{-2})(fp + hq).
$$

Since $p = S_n(\beta) - v_{22} S_{n-1}(\beta)$ and $q = v_{12} S_{n-1}(\beta)$ we have
\begin{eqnarray*}
ep + gq &=& e (S_n(\beta) - v_{22} S_{n-1}(\beta)) + g v_{12} S_{n-1}(\beta) = e S_n(\beta) + (g v_{12} - e v_{22})S_{n-1}(\beta), \\
fp + hq &=&  f (S_n(\beta) - v_{22} S_{n-1}(\beta)) + h v_{12} S_{n-1}(\beta) = f S_n(\beta) + (h v_{12} - f v_{22})S_{n-1}(\beta).
\end{eqnarray*}
With $v_{22} = t^{-1}eg+t^{-2}eh-fg$ and $v_{12} = -t^{-1}e((t-t^{-1})f-e)$, we have
\begin{eqnarray*}
g v_{12} - e v_{22} &=& t^{-2} e(fg-eh) = -  t^{-2}e,\\
h v_{12} - f v_{22} &=& t^{-1}(eh-fg) (ft-e) = t^{-1}(ft-e).
\end{eqnarray*}
Hence 
\begin{eqnarray*}
R_{n,m} &=& \big[ (t^2+t^{-2}-u-1)e - t^{-1}(u+1-t^{-2}) f \big] S_n(\beta) \\
  && \qquad - \, t^{-2} \big[ (t^2+t^{-2}-u-1)e - (u+1-t^{-2}) (ft-e) \big] S_{n-1}(\beta) \\
  &=& \big[ (t^2+t^{-2}-u-1)e - t^{-1}(u+1-t^{-2}) f \big] S_n(\beta)\\
   && \qquad - \, \big[ e - t^{-1} (u+1-t^{-2}) f \big] S_{n-1}(\beta) .
\end{eqnarray*}
Then, with $e = S_m(\alpha) - (u+1)S_{m-1}(\alpha)$ and $f = t(u-t^2+1)S_{m-1}(\alpha)$, we obtain
\begin{eqnarray*}
R_{n,m} &=& \big[(t^2 + t^{-2} - u -1) S_m(\alpha) - S_{m-1}(\alpha) \big] S_n(\beta) \\
&& \qquad + \, \big[(u^2-(t^2+t^{-2}-3)u + t^2+t^{-2}-3)S_{m-1}(\alpha) - S_m(\alpha) \big] S_{n-1}(\beta).
\end{eqnarray*}

Since $\alpha = u^2-u(t^2+t^{-2}-2)+2$ we have
\begin{eqnarray*}
&& (u^2-(t^2+t^{-2}-3)u - t^2 - t^{-2}+3)S_{m-1}(\alpha) - S_m(\alpha) \\
&=& (\alpha + u - t^2- t^{-2} + 1) S_{m-1}(\alpha) - S_m(\alpha) \\
&=& S_{m-2}(\alpha) - (t^2 + t^{-2} - u -1) S_{m-1}(\alpha).
\end{eqnarray*}
The formula for $R_{n,m}(x,y)$ follows, since $t^2 + t^{-2} = x^2-2$ and $u=y-2$.
\end{proof}

\section{Properties of the Riley polynomial} \label{lem}

In this section we will prove some properties of the Chebychev polynomials of the second kind and Riley polynomials of the rational knots $C(2n+1, 2m,2)$. We will make use of these properties in Section \ref{LO}. 

\subsection{Chebychev polynomials} 

Recall that $S_k(z)$'s are the Chebyshev polynomials  defined by $S_0(z)=1$, $S_1(z)=z$ and $S_{k}(z) = z S_{k-1}(z) - S_{k-2}(z)$ for all integers $k$. 

\begin{lemma} \label{chev} The followings hold true:
       \begin{enumerate}
		\item For $k \ge 0$, the polynomial $S_k(z)$ has degree $k$ and leading term $z^k$.
		\item For $k \ge 1$ we have $S_{-k}(z) = -S_{k-2}(z)$. 
		\item $S_k(\pm 2) = (\pm 1)^k (k+1)$ and $S_{k}(z) = (t^{k+1} - t^{-k-1})/(t - t^{-1})$ if $z = t + t^{-1} \not= \pm 2$. In particular we have $S_k(2\cos\theta) = \frac{\sin(k+1)\theta}{\sin \theta}$ for $\theta \in \BR \setminus \BZ\pi$. 
	\end{enumerate}
\end{lemma} 

\begin{proof}
All the equalities in the lemma can be proved by induction on $k \ge 0$. For the last equality, note that $2\cos\theta = e^{i\theta} + e^{-i\theta}$. 
\end{proof}

\begin{lemma} \label{root}
Suppose $k \ge 1$. Then the polynomial $S_k (z) - S_{k-1} (z)=0$ has $k$ roots given by $z = 2 \cos \theta_j$ where $\theta_j = \frac{(2j -1) \pi}{2k + 1}$ for $1 \leq j \leq k$. Moreover, at $z = 2 \cos \theta_j$ we have $(-1)^{j+1} S_k(z) > 0$.
\end{lemma}

\begin{proof}
By Lemma \ref{chev}(3), at $z = 2 \cos \theta_j$ we have 
	\begin{align*}
	S_k(z) &= \frac{\sin (k+1)\theta_j}{\sin \theta_j} =  \frac{\sin ((j-\frac{1}{2})\pi  + \frac{\theta_j}{2})}{\sin \theta_j} = \frac{(-1)^{j+1} \cos \frac{\theta_j}{2}}{\sin \theta_j} = \frac{(-1)^{j+1}}{2\sin \frac{\theta_j}{2}}, \\
	S_{k-1}(z) &= \frac{\sin k \theta_j}{\sin \theta_j} =  \frac{\sin ((j-\frac{1}{2})\pi  - \frac{\theta_j}{2})}{\sin \theta_j} = \frac{(-1)^{j+1} \cos \frac{\theta_j}{2}}{\sin \theta_j} = \frac{(-1)^{j+1}}{2\sin \frac{\theta_j}{2}}.
	\end{align*}
This implies that $S_k (z) = S_{k-1} (z)$ and $(-1)^{j+1} S_k(z) > 0$. Hence the polynomial $S_k (z) - S_{k-1} (z)$ has at least $k$ roots given by  $z = 2 \cos \theta_j$ for $1 \leq j \leq k$. Since the degree of $S_k (z) - S_{k-1} (z)$ is exactly $k$, those are all the roots. 
\end{proof}

\subsection{Riley poynomial} 

Recall from Propositions \ref{ab} and \ref{R} that 
\begin{eqnarray*}
\alpha &=& y^2 - x^2 y + 2x^2-2,\\
\beta &=& 2 + (x^2-y-2)\big( S_m(\alpha)+(1-y)S_{m-1}(\alpha) \big)^2.
\end{eqnarray*}
and
	$$
R_{n,m}(x,y) = \big( (x^2-y-1)S_m(\alpha)-S_{m-1}(\alpha) \big) S_n(\beta) - \big( (x^2-y-1)S_{m-1}(\alpha)-S_{m-2}(\alpha) \big) S_{n-1}(\beta).
$$

\begin{lemma} \label{limsign}
	For any fixed number $x \in \BR$, we have 
	$$
	\lim_{y \to \infty} (-1)^n R_{n,m}(x,y) =
	\begin{cases}  
	- \infty &\mbox{if } n \ge 1, m \ge 1, \\ 
         \infty &\mbox{if } n \ge 1, m \le -1, \\
        \infty &\mbox{if } n \le -1, m \ge 1, \\ 
        -\infty &\mbox{if } n \le -1, m \le -1. 
        \end{cases}
        $$
\end{lemma}

\begin{proof}
	We will prove the lemma for the case $n \ge 1$. The case $n \le -1$ is proved similarly. 
	
	By Lemma \ref{chev}(1), when $k \ge 0$ the leading term of $S_{k}(z)$ is $z^k$. Moreover, by Lemma \ref{chev}(2) we have $S_{-1}(z)=0$ and $S_{-k}(z) = -S_{k-2}(z)$ for $k \ge 2$. Hence when $k \ge 2$ the leading term of $S_{-k}(z)$ is $-z^{k-2}$. 
	
	Fix $x \in \BR$. Since the leading term of $\alpha$ is $y^2$, we see that the leading term of $(x^2-y-1)S_m(\alpha)-S_{m-1}(\alpha)$ is 
	$
	\begin{cases}  
        -y^{2m+1}  &\mbox{if } m \ge 1, \\ 
        y^{-2(m+1)}  &\mbox{if } m \le -1.
        \end{cases}
        $ 
        
        Similarly, the leading term of $(x^2-y-1)S_{m-1}(\alpha)-S_{m-2}(\alpha)$ is 
	$
	\begin{cases}  
        -y^{2m-1}  &\mbox{if } m \ge 1, \\ 
        y^{-2m}  &\mbox{if } m \le -1,
        \end{cases}
        $
        and the leading term of $\beta = 2 + (x^2-y-2)\big( S_m(\alpha)+(1-y)S_{m-1}(\alpha) \big)^2$ is $-y^{|4m+1|}$.
        
        Since $n \ge 1$, the leading term of $S_n(\beta)$ is equal to that of $\beta^n$, which is $(-y^{|4m+1|})^n$. 
        Similarly, the leading term of $S_{n-1}(\beta)$ is equal to that of $\beta^{n-1}$, which is $(-y^{|4m+1|})^{n-1}$. 
        
        If $m \ge 1$, the leading term of $R_{n,m}(x,y)$ is equal to that of $\big( (x^2-y-1)S_m(\alpha)-S_{m-1}(\alpha) \big) S_n(\beta)$, which is $-y^{2m+1} (-y^{4m+1})^n =  (-1)^{n+1} y^{2m+1 + (4m+1) n}$. 
        
        If $m \le -1$, the leading term of $R_{n,m}(x,y)$ is equal to that of $\big( (x^2-y-1)S_m(\alpha)-S_{m-1}(\alpha) \big) S_n(\beta)$, which is $y^{-2(m+1)} (-y^{-4m-1})^n =  (-1)^{n} y^{-2(m+1) + (-4m-1) n}$. 
        
        This completes the proof of the lemma for the case $n \ge 1$. 
	\end{proof}
	
\begin{lemma} \label{simplify}
Suppose $S_n(\beta) = S_{n-1}(\beta)$. Then 
$$
R_{n,m}(x,y)= (x^2 - y - 2) [S_m (\alpha) + (1 - y) S_{m-1} (\alpha)] S_n (\beta).
$$
\end{lemma}

\begin{proof}
If $S_n(\beta) = S_{n-1}(\beta)$ then
	\begin{align*}
		R_{n,m}(x,y) &= [(x^2 - y - 1) S_m (\alpha) - S_{m-1} (\alpha) - (x^2 - y - 1) S_{m-1} (\alpha) + S_{m-2} (\alpha)] S_n (\beta) \\
		&= [(x^2 - y - 1) (S_m (\alpha) - S_{m-1} (\alpha) )+ (\alpha -1)S_{m-1} (\alpha)  - S_{m} (\alpha)]  S_n (\beta) \\
					&= [(x^2 - y - 1) (S_m (\alpha) - S_{m-1} (\alpha)) \\ 
					&\hspace{3cm} - (\alpha - 1) (S_m (\alpha) - S_{m-1} (\alpha)) + (\alpha - 2) S_m (\alpha)] S_n (\beta) \\
					&= [(\alpha - 2) S_m (\alpha) + (x^2 - y - \alpha) (S_m (\alpha) - S_{m-1} (\alpha))] S_n (\beta).
	\end{align*}

Since $\alpha = y^2 - x^2 y + 2x^2-2$, we have $\alpha - 2 = (y-2)(y+2-x^2)$ and $x^2-y-\alpha = (1-y)(y+2-x^2)$. Hence 
\begin{align*}
R_{n,m}(x,y) &= (y+2-x^2) [(y-2) S_m(\alpha) + (1-y) (S_m (\alpha) - S_{m-1} (\alpha))] \\
	&= (x^2 - y - 2) [S_m (\alpha) + (1 - y) S_{m-1} (\alpha)] S_n (\beta).
\end{align*}
\end{proof}

\begin{lemma} \label{y=2}
Suppose $y = 2$. Then $\alpha = 2$, $\beta= x^2-2$ and
$$
R_{n,m}(x,y)= m (\beta-2) (S_n(\beta) - S_{n-1}(\beta)) + S_{n+1}(\beta) -  S_n(\beta).
$$
Moreover, if we also have $x=2\cos\theta$ for some $\theta \in \BR \setminus \BZ\frac{\pi}{2}$ then 
$$
R_{n,m}(x,y) = m ( 2\cos 2\theta-2) \frac{\cos (2n+1)\theta}{\cos \theta} + \frac{\cos (2n+3)\theta}{\cos\theta}.
$$
\end{lemma}

\begin{proof}
Since $y = 2$ we have $\alpha = y^2 - x^2 y + 2x^2-2=2$ and so $S_k(\alpha) = k+1$ for all $k \in \BZ$. This implies that $\beta =2 + (x^2-y-2)\big( S_m(\alpha)+(1-y)S_{m-1}(\alpha) \big)^2 = x^2 - 2$ and
	\begin{align*}
		R_{n,m}(x,y) & =[ (x^2-3)S_m(\alpha)-S_{m-1}(\alpha) ] S_n(\beta) - [(x^2-3)S_{m-1}(\alpha)-S_{m-2}(\alpha)] S_{n-1}(\beta)\\
		&= [(\beta-1)(m+1) - m]S_n(\beta) - [(\beta-1)m- (m-1)]S_{n-1}(\beta) \\
		&= m (\beta-2) (S_n(\beta) - S_{n-1}(\beta)) + (\beta-1) S_n(\beta) - S_{n-1}(\beta) \\
		&= m (\beta-2) (S_n(\beta) - S_{n-1}(\beta)) + S_{n+1}(\beta) -  S_n(\beta).
	\end{align*}
	
	 If we also have $x=2\cos\theta$ for some $\theta \in \BR \setminus \BZ\frac{\pi}{2}$ then $\beta= x^2-2=2\cos2\theta$. Since $2\theta \in \BR \setminus \BZ \pi$, by Lemma \ref{chev}(3) we have  $S_k(\beta) = \frac{\sin (2k+2) \theta}{\sin 2\theta}$ for all $k \in \BZ$. Hence
	\begin{align*}
		R_{n,m}(x,y) &= m ( 2\cos 2\theta-2) \frac{\sin (2n+2) \theta- \sin 2n\theta}{\sin 2\theta} + \frac{\sin (2n+4)\theta- \sin (2n+2)\theta}{\sin 2\theta} \\
		&= m ( 2\cos 2\theta-2) \frac{\cos (2n+1)\theta}{\cos \theta} + \frac{\cos (2n+3)\theta}{\cos\theta}.
	\end{align*}

\end{proof}

\section{ Left orderability of cyclic branched covers} \label{LO}

In this section we will study real roots of the Riley polynomial of the rational knot $C(2n+1,2m,2)$. We then use the properties of real roots to make a conclusion about the left orderability of the fundamental groups of the cyclic branched covers of  $C(2n+1,2m,2)$. 

\subsection{Real roots of the Riley polynomial}

There are four cases to consider: when both $n$ and $m$ are the same sign and when $n$ and $m$ are opposite in sign. We will carefully study real roots of the Riley polynomial for the cases $n \geq 1$, $m \geq 1$ and $n \geq 2$, $m \leq -1$. The other two cases are similar. Additionally, the cases $n = 1$, $m \leq -1$ and $n = -2$, $m \geq 1$ differ slightly from the others. So we will consider them separately.

\subsubsection{} \textbf{Case $\bf{1}$:} $n \geq 1$, $m \geq 1$.

\begin{lemma} \label{asign1}
	Suppose $y > 2$ and $x \in \BR$ satisfies $x^2 \le 4$. Then $$S_m (\alpha) + (1 - y) S_{m-1} (\alpha) > 0.$$
\end{lemma}

\begin{proof}
	Let $u = y - 2$, then $\alpha = 2 + u^2 + (4 - x^2)u$. Since $u > 0$ and $4 \geq x^2$, we have $\alpha \ge 2 + u^2 := \alpha_0.$ Choose $v > 1$ such that $u = v - v^{-1}$, and so $\alpha_0 = v^2 + v^{-2}$. Then
	\begin{eqnarray*}
		\frac{S_m (\alpha_0)}{S_{m-1} (\alpha_0)} - u - 1 &=& \frac{v^{2m+2} - v^{-2m-2}}{v^{2m} - v^{-2m}} - (v-v^{-1})-1\\  &=& \frac{(v^2 - 1) (v+1+ v^{1 + 4m} (v-1))}{v^2(v^{4m} -1)} > 0.
	\end{eqnarray*}
	
	Since $S_m (\alpha) / S_{m-1} (\alpha)$ is increasing in $\alpha \in (2, \infty)$ and $\alpha \ge \alpha_0$, we have
	\[
		\frac{S_m (\alpha)}{S_{m-1} (\alpha)} \ge \frac{S_m (\alpha_0)}{S_{m-1} (\alpha_0)} > u + 1 = y-1.
	\]
	Therefore $S_m (\alpha) + (1 - y) S_{m-1} (\alpha) > 0$.
\end{proof}

\begin{proposition} \label{prop1}
	Suppose $x \in \BR$ satisfies $4 \geq x^2 > 2 + 2 \cos \frac{(2n-1)\pi}{2n + 1}$. Then $R_{n,m} (x,y) = 0$ has at least one real solution $y > 2$.
\end{proposition}

\begin{proof}	
	By Lemma \ref{limsign} we have $(-1)^{n} R_{n,m}(x,y) \to  -\infty$ as $y \to \infty$. To prove the proposition, it suffices to show that there exists $y_0 \ge 2$ such that $(-1)^{n} R_{n,m}(x,y_0)>0$. 
	
	If we let $y = 2$, then by Lemma \ref{y=2} we have 
	$\beta = x^2 - 2 >  2 \cos \frac{(2n-1)\pi}{2n + 1}$. Since the leading term of $\beta $ is $-y^{4m+1}$, we have $\beta \to - \infty$ as $y \to \infty$. Hence there exists $y_0 \in (2, \infty)$ such that $\beta = 2 \cos \frac{(2n-1)\pi}{2n + 1}$. Lemma \ref{root} implies that  $S_n(\beta) = S_{n-1}(\beta)$ and $(-1)^{n+1}S_n(\beta) >0$. Then by Lemma \ref{simplify} we have
	$$
	R_{n,m}(x,y_0) = (x^2 - y_0 - 2) [S_m (\alpha) + (1 - y_0) S_{m-1} (\alpha)] S_n (\beta).
	$$
	
	Since $y_0>2$ and $x \in \BR$ satisfies $4 \geq x^2$, by Lemma \ref{asign1} we have $S_m (\alpha) + (1 - y_0) S_{m-1} (\alpha)>0$. Note that $x^2- y_0 -2 < 0$ and the sign of $S_n (\beta)$ is $(-1)^{n+1}$. So the sign of $R_{n,m}(x,y_0)$ is $(-1)^n$, which means that $(-1)^n R_{n,m}(x,y_0) > 0$. 
\end{proof}

If we let $x = 2\cos\frac{\pi}{r}$ for some integer $r \ge 2$, then the condition $x^2 > 2 + 2 \cos \frac{\pi (2n-1)}{2n + 1}$ is equivalent to $\cos\frac{2\pi}{r} > \cos \frac{\pi (2n-1)}{2n + 1}$ and so $r > \frac{2(2n+1)}{2n-1} = 2 + \frac{4}{2n-1}$. This gives the following lower bound for $r$: 
\begin{itemize}
\item $r \ge 3$ if $n \ge 3$.
\item $r \ge 4$ if $n = 2$.
\item $r \ge 7$ if $n =1$.
\end{itemize}
In the case $n = 1$, we can slightly improve the lower bound for $r$.

\begin{proposition} \label{prop2}
	$R_{1,m}(2\cos\frac{\pi}{r},y) = 0$ has at least one real solution $y > 2$ in the following cases:
	\begin{enumerate}
		\item $r = 5,6$ when $m=1,2$.
		\item $r = 6$ when $m \geq 3$.
	\end{enumerate}
\end{proposition}

\begin{proof}
	Since $m \ge 1$, by Lemma \ref{limsign} we have $R_{1,m}(2\cos\frac{\pi}{r},y) \to \infty$ as $y \to \infty$. To prove the proposition, it suffices to check that $R_{1,m}(2\cos\frac{\pi}{r},2) <0$. 
	
	By Lemma  \ref{y=2} we have $R_{1,m}(2\cos\frac{\pi}{r},2) =  m ( 2\cos \frac{2\pi}{r}-2) \frac{\cos \frac{3\pi}{r}}{\cos \frac{\pi}{r}} + \frac{\cos \frac{5\pi}{r}}{\cos\frac{\pi}{r}}$.
	
	If $r=6$, then  $R_{1,m}(2\cos\frac{\pi}{r},2) =-1<0$.
	
	If $r=5$, then by a direct calculation we have 
	$$R_{1,m}(2\cos\frac{\pi}{r}, 2) = (5-2\sqrt{5})m + 1- \sqrt{5}
= \begin{cases} 
6-3\sqrt{5} &\mbox{if } m=1, \\ 
11-5\sqrt{5} &\mbox{if } m=2,
\end{cases}$$ 
which is negative when $m=1,2$.
\end{proof}

\subsubsection{}  \textbf{Case $\bf{2}$:} $n \geq 2$, $m \leq -1$.

\begin{proposition} \label{prop3}
	Suppose $x \in \BR$ satisfies $4 \geq x^2 > 2 + 2 \cos \frac{(2n-3)\pi}{2n + 1}$. Then $R_{n,m} (x,y) = 0$ has at least one real solution $y > 2$.
\end{proposition}

\begin{proof}
By Lemma \ref{limsign} we have $(-1)^{n} R_{n,m}(x,y) \to  \infty$ as $y \to \infty$. To prove the proposition, it suffices to show that there exists $y_0 \ge 2$ such that $(-1)^{n} R_{n,m}(x,y_0)<0$. 

If we let $y = 2$, then by Lemma \ref{y=2} we have 
	$\beta = x^2 - 2 >  2 \cos \frac{(2n-3)\pi}{2n + 1}$. Since the leading term of $\beta $ is $-y^{|4m+1|}$, we have $\beta \to - \infty$ as $y \to \infty$.  Hence there exists $y_0 > 2$ such that $\beta = 2 \cos \frac{(2n-3)\pi}{2n + 1}$. Lemma \ref{root} implies that  $S_n(\beta) = S_{n-1}(\beta)$ and $(-1)^{n}S_n(\beta) >0$. Then by Lemma \ref{simplify} we have
	\begin{align*}
		R_{n,m}(x,y_0) &=(x^2 - y_0 - 2) [S_m (\alpha) + (1 - y_0) S_{m-1} (\alpha)] S_n (\beta) \\
					&= (x^2 - y_0 - 2) [(y_0-1) S_{-m-1} (\alpha) - S_{-m-2}(\alpha)] S_n (\beta).
	\end{align*}
	
	 Since $y_0 > 2$ and $x^2 \le 4$, we have $\alpha = 2 + (y_0-2)(y_0+2-x^2) > 2$. This implies that 
	 $$
	 (y_0-1) S_{-m-1} (\alpha) - S_{-m-2}(\alpha) > S_{-m-1} (\alpha) - S_{-m-2}(\alpha) > 0.
	 $$
So the sign of $R_{n,m}(x,y_0)$ is $(-1)^{n+1}$, which means that  $(-1)^{n} R_{n,m}(x,y_0) < 0$. 
\end{proof}

If we let $x = 2\cos\frac{\pi}{r}$ for some integer $r \ge 2$, then the condition $x^2 > 2 + 2 \cos \frac{\pi (2n-3)}{2n + 1}$ becomes $\cos\frac{2\pi}{r} > \cos \frac{\pi (2n-3)}{2n + 1}$ and so $r > \frac{2(2n+1)}{2n-3} = 2 + \frac{8}{2n-3}$. 
This gives the following lower bound for $r$: 
\begin{itemize}
\item $r \ge 3$ if $n \ge 6$.
\item $r \ge 4$ if $n = 4, 5$.
\item $r \ge 5$ if $n =3$.
\item $r \ge 11$ if $n =2$.
\end{itemize}
 In the cases $n = 2,3,4,5$, we can slightly improve these lower bounds.

\begin{proposition} \label{prop4}
$R_{n,m}(2\cos\frac{\pi}{r},y) = 0$ has at least one real solution $y > 2$ in the following cases:
	\begin{enumerate}
		\item $r = 3$ when $n=4,5$.
		\item $r = 3, 4$ when $n=3$.
		\item $4 \leq r \leq 10$ when $n = 2$.
	\end{enumerate}
\end{proposition}

\begin{proof}
Since $m \le -1$, by Lemma \ref{limsign} we have $(-1)^{n} R_{n,m}(2\cos\frac{\pi}{r},y) \to  \infty$ as $y \to \infty$. To prove the proposition, it suffices to check that $(-1)^{n} R_{n,m}(2\cos\frac{\pi}{r}, 2)<0$. 

By Lemma  \ref{y=2} we have 
$$
R_{n,m}(2\cos\frac{\pi}{r}, 2) = m ( 2\cos \frac{2\pi}{r}-2) \frac{\cos \frac{(2n+1)\pi}{r}}{\cos \frac{\pi}{r}} + \frac{\cos \frac{(2n+3)\pi}{r}}{\cos\frac{\pi}{r}}.
$$

If $r=3$ then 
$R_{n,m}(2\cos\frac{\pi}{r}, 2) = -6m \cos \frac{(2n+1)\pi}{3} + 2\cos \frac{(2n+3)\pi}{3}
= \begin{cases} 
 -2 -3m &\mbox{if } n =3, \\ 
 1 + 6m &\mbox{if } n =4, \\
  1-3m&\mbox{if } n =5,
\end{cases}
$
which has the sign $(-1)^{n+1}$ when $n=3,4,5$.

If $r =4$  then $R_{n,m}(2\cos\frac{\pi}{r}, 2) = -2\sqrt{2} m \cos \frac{(2n+1)\pi}{4} + \sqrt{2} \cos \frac{(2n+3)\pi}{4}
= \begin{cases} 
1 +2m &\mbox{if } n =2, \\ 
 1-2m &\mbox{if } n =3,
\end{cases}$ which has the sign $(-1)^{n+1}$ when $n=2,3$.

If $n=2$ and $5 \le r \le 10$ then by direct calculations we have
\begin{eqnarray*}
R_{n,m}(2\cos\frac{\pi}{r}, 2) &=& m ( 2\cos \frac{2\pi}{r}-2) \frac{\cos \frac{5\pi}{r}}{\cos \frac{\pi}{r}} + \frac{\cos \frac{7\pi}{r}}{\cos\frac{\pi}{r}} \\
&\le& -( 2\cos \frac{2\pi}{r}-2) \frac{\cos \frac{5\pi}{r}}{\cos \frac{\pi}{r}} + \frac{\cos \frac{7\pi}{r}}{\cos\frac{\pi}{r}} \\
&<& 0.
\end{eqnarray*}
So $R_{n,m}(2\cos\frac{\pi}{r}, 2)$ has the sign $(-1)^{n+1}$. 
\end{proof}

The cases $n \leq -3$, $m \geq 1$ and $n \leq -2$, $m \leq -1$ can be proved in a similar way as Cases $1$ and $2$ above. However, this method of proof does not include the cases $n = 1$, $m \leq -1$ and $n = -2$, $m \geq 1$. We will consider them next.

\subsubsection{} \textbf{Case $\bf{3}$:} $n=1$, $m \leq -1$.

\begin{proposition} \label{n1sol}
	$R_{1,m}(2\cos\frac{\pi}{r},y) = 0$ has at least one real solution $y > 2$ in the following cases:
	\begin{enumerate}
		\item $r \geq 7$ when $m \leq -4$.
		\item $r \geq 8$ when $m = -2,-3$.
		\item $r \geq 9$ when $m = -1$.
	\end{enumerate}
\end{proposition}

\begin{proof}

Since $m \le -1$, by Lemma \ref{limsign} we have $R_{1,m}(2\cos\frac{\pi}{r},y) \to  -\infty$ as $y \to \infty$. To prove the proposition, it suffices to check that $R_{1,m}(2\cos\frac{\pi}{r}, 2)>0$. 

By Lemma  \ref{y=2} we have 
\begin{align*}
	R_{1,m}(2\cos\frac{\pi}{r}, 2)  &= m(\beta - 2)(S_1(\beta) - S_{0}(\beta)) + S_{2}(\beta) - S_1(\beta) \\
	&= m(\beta - 2)(\beta - 1) + \beta^2 - \beta - 1,
\end{align*}
where $\beta = x^2 - 2 = 2 \cos \frac{2\pi}{r}$. 

For $m \le -2$, the inequality $R_{1,m}(2\cos\frac{\pi}{r}, 2)>0$ is equivalent to 
$$
\frac{3m + 1 -\sqrt{m^2 + 2m + 5}}{2(m+1)} > \beta = 2 \cos \frac{2\pi}{r} > \frac{3m + 1 + \sqrt{m^2 + 2m + 5}}{2(m+1)}. 
$$
Note that $\frac{3m + 1 -\sqrt{m^2 + 2m + 5}}{2(m+1)} > \frac{3m + 3 -\sqrt{m^2 + 2m + 1}}{2(m+1)} =2$.  Hence we only need
$$2 \cos \frac{2\pi}{r} > \frac{3m + 1 + \sqrt{m^2 + 2m + 5}}{2(m+1)} =  \frac{3}{2} - \left( \frac{1}{m+1} + \sqrt{\frac{1}{4} + \frac{1}{(m+1)^2}} \, \right).$$ 

For $m=-2$, we have $2 \cos \frac{2\pi}{r} > \frac{5-\sqrt{5}}{2}$ if $r \ge 8$. 

For $m=-3$, we have $2 \cos \frac{2\pi}{r} >1-  \frac{\sqrt{2}}{4}$ if $r \ge 8$. 

For $m \le -4$, we have 
$$
\frac{1}{m+1} + \sqrt{\frac{1}{4} + \frac{1}{(m+1)^2}} = \frac{1/4}{\frac{1}{-(m+1)}+\sqrt{\frac{1}{4} + \frac{1}{(m+1)^2}} } \ge \frac{1/4}{ \frac{2+\sqrt{13}}{6}} = \frac{\sqrt{13}-2}{6}.
$$
So we only need $2 \cos \frac{2\pi}{r} > \frac{3}{2} - \frac{\sqrt{13}-2}{6}$. This holds true if $r \ge 7$.

Finally, when $m=-1$ we have $R_{1,m}(2 \cos \frac{\pi}{r}, 2) = 4 \cos \frac{2\pi}{r} - 3 < 0$ if $r \geq 9$.
\end{proof}

\subsubsection{}  \textbf{Case $\bf{4}$:} $n = -2$, $m \geq 1$.

\begin{lemma} \label{=1}
Suppose $m\ge 1$ and $x \in \BR$ such that $4 \ge x^2 > 2 + \frac{2m}{m+1}$. Then there exists a unique $y_0>2$ such that $S_m(\alpha) + (1-y_0)S_{m-1}(\alpha) =1$. 
\end{lemma}

\begin{proof}
Consider $y>2$. Note that $S_m(\alpha) + (1-y)S_{m-1}(\alpha) >0$ (by Lemma \ref{asign1}) and  $S^2_m(\alpha) + S^2_{m-1}(\alpha) - \alpha S_m(\alpha) S_{m-1}(\alpha)=1$ (by Lemma \ref{che}). This implies that the equation  $S_m(\alpha) + (1-y_0)S_{m-1}(\alpha) =1$ is equivalent to $$(S_m(\alpha) + (1-y)S_{m-1}(\alpha))^2 = S^2_m(\alpha) + S^2_{m-1}(\alpha) - \alpha S_m(\alpha) S_{m-1}(\alpha).$$ 

With $\alpha=2+(y-2)(2+y-x^2)$, the above equation becomes $(y-2)S_{m-1}(y) [(x^2-y)S_m(\alpha)-yS_{m-1}(\alpha)]=0$, i.e. $(x^2-y)S_m(\alpha)-yS_{m-1}(\alpha)=0$. We have shown that for $y>2$, the equation $S_m(\alpha) + (1-y)S_{m-1}(\alpha) =1$ is equivalent to $(x^2-y)S_m(\alpha)-yS_{m-1}(\alpha) = 0$, which is also equivalent to $P(y):=(y^2-y x^2 + x^2)S_{m-1}(\alpha)+y-x^2=0$. 

It is easy to see that $P(y)$ is an increasing function on $(2,\infty)$ and $\lim_{y \to \infty} P(y)=\infty$. Moreover $P(2)=(4-x^2)m+2-x^2 = 2m - (x^2-2)(m+1) < 0$, since $x^2-2 > \frac{2m}{m+1}$. Hence there exists a unique $y_0>2$ such that $S_m(\alpha) + (1-y)S_{m-1}(\alpha) =1$.
\end{proof}

\begin{proposition} \label{n-2sol}
	$R_{-2,m}(2\cos\frac{\pi}{r},y) = 0$ has at least one real solution $y > 2$ if $r \geq 7$ and $m \ge 6$.
\end{proposition}

\begin{proof}
Suppose $r \geq 7$ and $m \ge 6$. By Lemma \ref{limsign} we have $R_{-2,m}(2\cos\frac{\pi}{r},y) \to  \infty$ as $y \to \infty$.  To prove the proposition,  it suffices to show that there exists $y_0 \ge 2$ such that $R_{-2,m}(2\cos\frac{\pi}{r}, y_0)<0$. 

We first consider the case $y_0=2$. By Lemma  \ref{y=2} we have 
\begin{eqnarray*}
R_{-2,m}(2\cos\frac{\pi}{r},2) &=& m (\beta-2) (S_{-2}(\beta) - S_{-3}(\beta)) + S_{-1}(\beta) -  S_{-2}(\beta) \\
&=& m (\beta-2)(-1+\beta)+1,
\end{eqnarray*}
where $\beta = x^2 - 2 = 2 \cos \frac{2\pi}{r}$. 

For $m \ge 6$, the inequality $R_{-2,m}(2\cos\frac{\pi}{r},2) <0$ is equivalent to 
$$
\frac{3 -\sqrt{1-4/m}}{2} < 2 \cos \frac{2\pi}{r} < \frac{3 + \sqrt{1-4/m}}{2}. 
$$
By a direct calculation we have $\frac{3 -\sqrt{1-4/m}}{2} \le \frac{3 -\sqrt{1/3}}{2} < 2 \cos \frac{2\pi}{7}  \le 2 \cos \frac{2\pi}{r}$. Hence $R_{-2,m}(2\cos\frac{\pi}{r},2) <0$ if $2 \cos \frac{2\pi}{r} < \frac{3 + \sqrt{1-4/m}}{2}$. 

It remains to consider the case $2 \cos \frac{2\pi}{r} \ge \frac{3 + \sqrt{1-4/m}}{2}$. Note that $\frac{3 + \sqrt{1-4/m}}{2} > \frac{2m}{m+1}$. In fact, this is equivalent to $\frac{m-3}{m+1} < \sqrt{1-\frac{4}{m}}$ which holds true since $( \frac{m-3}{m+1})^2 - (1-\frac{4}{m}) = \frac{-4(m^2-4m-1)}{m(m+1)^2}<0$. Hence $2 \cos \frac{2\pi}{r} > \frac{2m}{m+1}$.  By Lemma \ref{=1} there exists $y_0>2$ such that $S_m(\alpha) + (1-y)S_{m-1}(\alpha) =1$. 
We claim that $R_{-2,m}(2 \cos \frac{\pi}{r}, y_0)<0$.

\begin{figure}[h]
	\centering
	\includegraphics[scale=.75]{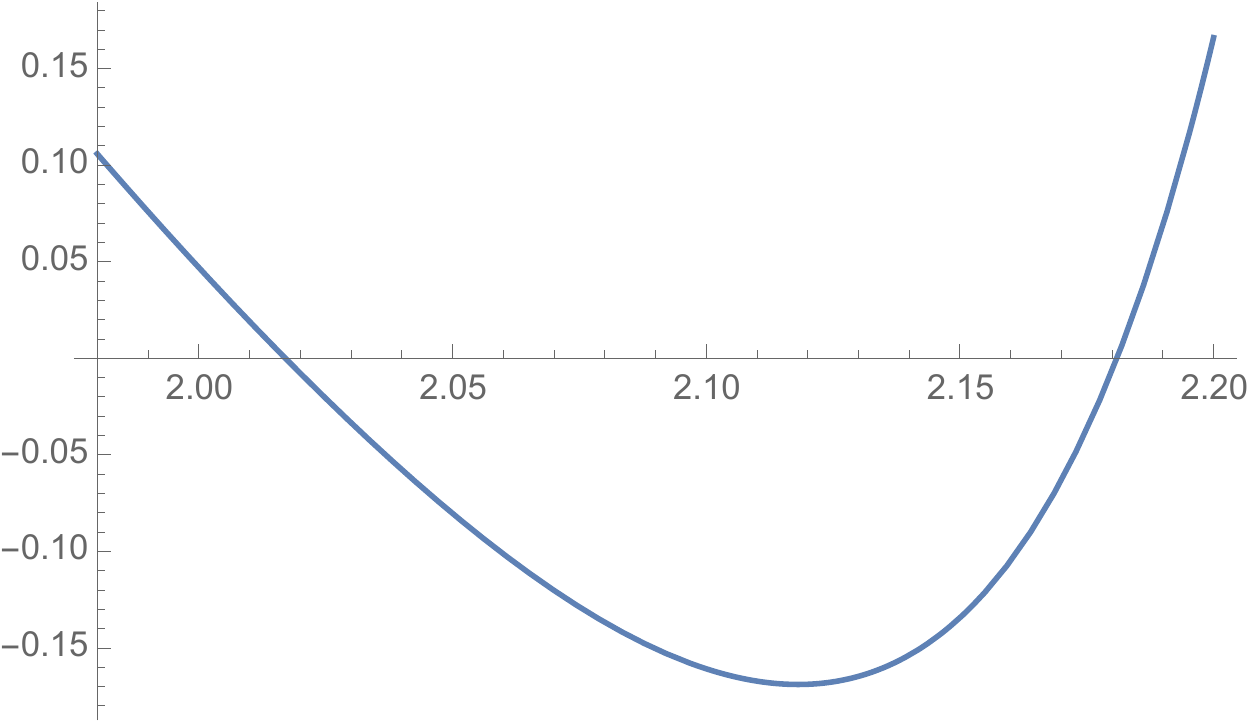}
	\caption{Plot of $R_{-2,m}(2 \cos \frac{\pi}{r}, y)$ for $m = 6$ and $r = 14$.}
	\label{fig:Example1}
\end{figure}

When $y=y_0$ we have $\beta = 2 + (x^2-y_0-2)\big( S_m(\alpha)+(1-y_0)S_{m-1}(\alpha) \big)^2  = x^2-y_0$. Since $S_{m-2}(\alpha) = \alpha S_{m-1}(\alpha) - S_m(\alpha)$ and $S_m(\alpha) = 1 +(y_0 -1)S_{m-1}(\alpha)$ we get
\begin{eqnarray*}
R_{-2,m}(2 \cos \frac{\pi}{r},y_0) &=& ( (x^2-y_0-1)S_m(\alpha)-S_{m-1}(\alpha)) S_{-2}(\beta) \\
&& \qquad \qquad - \,  ( (x^2-y_0-1)S_{m-1}(\alpha)-S_{m-2}(\alpha) ) S_{-3}(\beta) \\
&=& - ((\beta-1) S_m(\alpha) - S_{m-1}(\alpha))+ \beta  ((\beta-1) S_{m-1}(\alpha) - S_{m-2}(\alpha)) \\
&=& (\beta^2 - \beta +1 - \beta \alpha) S_{m-1}(\alpha)+ S_m(\alpha) \\
&=&  (\beta^2 - \beta +1 - \beta \alpha + y_0-1) S_{m-1}(\alpha)+ 1.
\end{eqnarray*}

With $\alpha = 2 + (y_0-2)(y_0+2-x^2) =2+(y_0-2)(2-\beta)$, by a direct calculation  we have $\beta^2 - \beta +1 - \beta \alpha + y_0-1 =-(\beta-1)(\beta + y_0 - \beta y_0)$. 

By Lemma \ref{=1} we have  $P(y_0)=(y_0^2-y_0 x^2 + x^2)S_{m-1}(\alpha)+y_0-x^2=0$, which implies that
$
S_{m-1}(\alpha)= \frac{x^2 - y_0}{y_0^2-y_0 x^2 + x^2} = \frac{\beta}{\beta+ y_0 - \beta y_0}.
$ 
Hence
$$
R_{-2,m}(2 \cos \frac{\pi}{r},y_0) = -(\beta-1)(\beta + y_0 - \beta y_0) \frac{\beta}{\beta+ y_0 - \beta y_0}+1 = 1+\beta-\beta^2. 
$$
Hence $R_{-2,m}(2 \cos \frac{\pi}{r},y_0)< 0$ if we can show that
$\beta > \frac{1+\sqrt{5}}{2}$, i.e. $y_0 < x^2 - \frac{1+\sqrt{5}}{2}$. 

Let $z := x^2 -2  = 2 \cos \frac{2\pi}{r}\ge \frac{3 + \sqrt{1-4/m}}{2}$. Then $x^2 - \frac{1+\sqrt{5}}{2} = z + \frac{3-\sqrt{5}}{2} \ge \frac{3 + \sqrt{1-4/m}}{2}+ \frac{3-\sqrt{5}}{2}>2$ for $m \ge 6$. Since $P$ is an increasing function on $(2,\infty)$ with a unique root $y_0>2$, we conclude that $y_0 < x^2 - \frac{1+\sqrt{5}}{2} = z + \frac{3-\sqrt{5}}{2}$ is equivalent to $P(z + \frac{3-\sqrt{5}}{2})>0$. 

At $y=y_1:=z + \frac{3-\sqrt{5}}{2}$ we have $y^2-yx^2+x^2=\frac{\sqrt{5}-1}{2}(\sqrt{5}-z)$ and $\alpha = \frac{5-\sqrt{5}+(3-\sqrt{5})z}{2}$. This implies that $P(y_1)=\frac{\sqrt{5}-1}{2}(\sqrt{5}-z) S_{m-1}(\alpha) - \frac{1+\sqrt{5}}{2}$. 

If $m \ge 7$ then $z \ge \frac{3 + \sqrt{1-4/m}}{2} \ge \frac{3 + \sqrt{3/7}}{2}>1.827$ and so $\alpha = \frac{5-\sqrt{5}+(3-\sqrt{5})z}{2} > 2.0798$. This implies that $S_{m-1}(\alpha) > S_6(2.0798) > 12$. Hence 
$$P(y_1) > \frac{\sqrt{5}-1}{2}(\sqrt{5}-2) (12) - \frac{1+\sqrt{5}}{2} = \frac{83-37\sqrt{5}}{2}>0.$$

\begin{figure}[h]
	\centering
	\includegraphics[scale=.75]{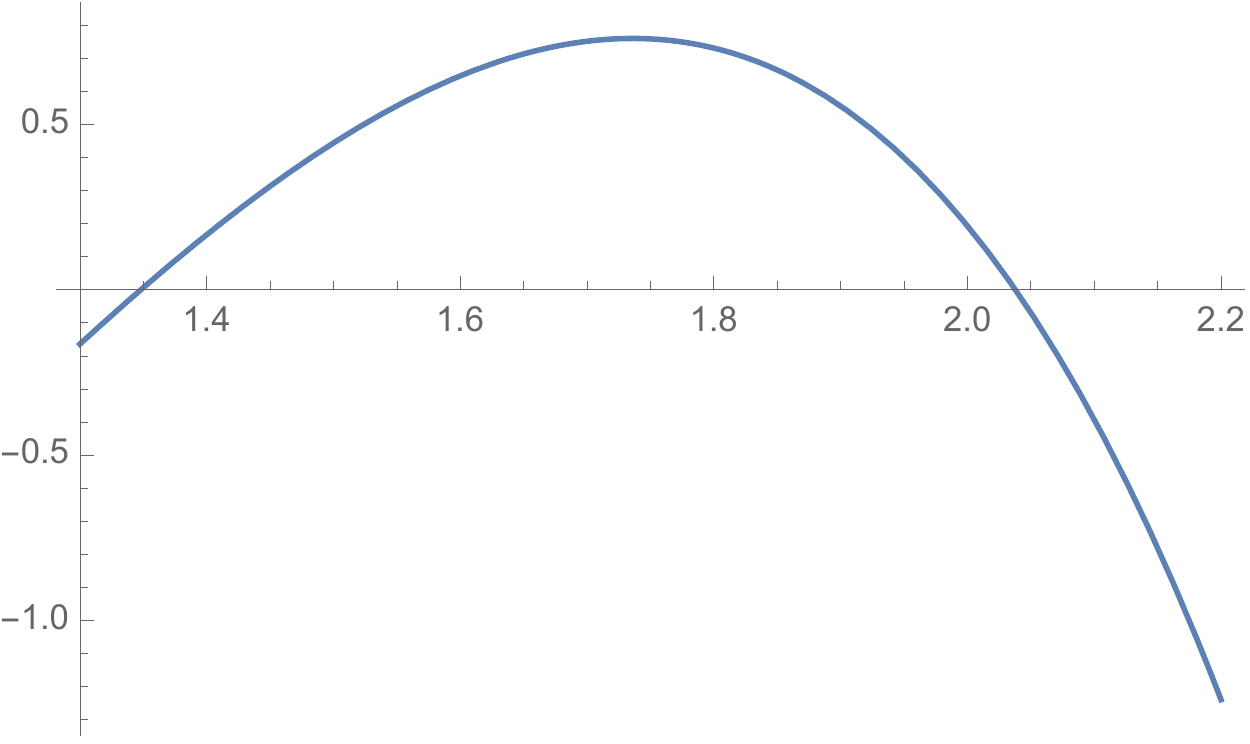}
	\caption{Plot of $P(y) = 0$ for $m = 6$.}
	\label{fig:Example2}
\end{figure}

If $m=6$ then $z \ge \frac{3 + \sqrt{1-4/m}}{2} \ge \frac{3 + \sqrt{2/3}}{2}>1.788$. Since $P(y_1)=\frac{\sqrt{5}-1}{2}(\sqrt{5}-z) S_{5}(\alpha) - \frac{1+\sqrt{5}}{2}$ where $\alpha = \frac{5-\sqrt{5}+(3-\sqrt{5})z}{2}$, $P(y_1)$ is a polynomial in $z$ of degree 6 with negative leading coefficient. This polynomial has exactly 6 real roots and the two largest ones are approximately $1.34811$ and $2.03784$ (see Figure \ref{fig:Example2}). Since $z$ lies in the interval between these two roots, we conclude that $P(y_1)>0$. 
\end{proof}

\begin{remark}
In the case $n=-2$ and $m=5$, the bounds $8 \le r \le 26$ can be checked directly by Mathematica. By numerical experiments, $R_{-2,m}(2\cos\frac{\pi}{r},y)$ does not have any real root $y>2$ if $m=5$ and $r \ge 27$, or $1 \le m \le 4$ and $r \ge 1$. 
\end{remark}

\subsubsection{Conclusion}  We combine the results from all cases to get the following:

\begin{proposition} \label{allsol}
	$R_{n,m}(2 \cos\frac{\pi}{r},y) = 0$ has at least one real solution $y > 2$ if:
	\begin{enumerate}
		\item $r \ge 3$ when $n \ge 3$ or $n \le -4$.
		\item $r \ge 4$ when $n =2$ or $n =-3$.
		\item $r \ge 5$ when $n=1$ and $m=1,2$, or $n=-2$ and $m=-1$.
		\item $r \ge 6$ when $n=1$ and $m \ge 3$, or $n=-2$ and $m \le -2$.
		\item $r \geq 7$ when $n=1$ and $m \leq -4$, or $n = -2$ and $m \geq 6$.
		\item $r \geq 8$ when $n=1$ and $m = -2,-3$.
		\item $r \geq 9$ when $n=1$ and $m = -1$.
	\end{enumerate}
\end{proposition}

\subsection{Proof of Theorem \ref{thm2}}

Let $X^{(r)}_{K}$ denote the $r$-th cyclic branched cover of a knot $K$ in $S^3$. We will apply the following theorem.

\begin{theorem}[\cite{BGW, Hu}] \label{huthm}
	Given any prime knot $K$ in $S^3$, denote by $\mu$ a meridional element of the knot group $G(K)$. If there exists a nonabelian representation $\rho: G(K) \to \mathrm{SL}_2(\BR)$ such that $\rho(\mu^r) = \pm I$, then the fundamental group $\pi_1( X^{(r)}_{K} )$ is left orderable.
\end{theorem}

We are now ready to prove Theorem \ref{thm2}. We will prove the case when $n \ge 3$ or $n \le -4$. The other cases are proved similarly.

Suppose $r \in \BN$ satisfies $r \geq 3$ and let $x = 2 \cos \frac{\pi}{r}$. By Proposition \ref{allsol} there exists $y \in \BR$ with $y > 2$ such that $R_{n,m}(x,y) = 0$. Therefore, we have a nonabelian representation $\rho: G(K_{n,m}) \to \mathrm{SL}_2(\BC)$ of the form
\[
		\rho(a) = \begin{bmatrix} e^{i(\pi / r)} & 1 \\ 0 & e^{-i(\pi / r)} \end{bmatrix}, \qquad \rho(b) = \begin{bmatrix} e^{i(\pi / r)} & 0 \\ -u & e^{-i(\pi / r)} \end{bmatrix}.
\]
Note that $x = \text{tr} \rho(a) = 2 \cos \frac{\pi}{r}$.

Since $y = \text{tr} \rho (ab^{-1}) = 2 + u > 2$, we have $u > 0$. By \cite[page 786]{Kh}, the representation $\rho$ can be conjugated to an $\mathrm{SL}_2(\BR)$ representation, denoted by $\rho'$. Note that $\rho'(a^r) = -I$ since $\rho(a^r) = -I$. Therefore, by Theorem \ref{huthm}, $\pi_1( X^{(r)}_{K_{n,m}})$ is left-orderable.

\section*{Acknowledgements} 
The second author is partially supported by a grant from the Simons Foundation (\#354595).

\end{document}